\documentclass[12pt,a4paper]{amsart}
\usepackage{amssymb,amsmath,amsfonts,amsthm,color,url,graphicx,hyperref}
\hypersetup{%
  colorlinks,
  bookmarksopen,
  bookmarksnumbered,
  citecolor=blue,
  linkcolor=black,
  pdfstartview=FitH,
}

\newcommand{\qtext}[1]{\quad\text{#1}\quad}
\newcommand{\qand}{\qtext{and}}

\newtheorem{thm}{Theorem}
\newtheorem{prop}[thm]{Proposition}
\newtheorem{lem}[thm]{Lemma}

\newtheorem{cor}[thm]{Corollary}
\theoremstyle{definition}
\newtheorem{defn}[thm]{Definition}
\newtheorem{exa}[thm]{Example}
\newtheorem{rem}[thm]{Remark}
\theoremstyle{remark}
\newtheorem*{ack}{Acknowledgements}
\numberwithin{thm}{section}

\DeclareMathOperator*{\sotlim}{{\mbox{\scshape sot}}-lim}
\newcommand{\sotto}{\stackrel{\textsc{sot}}{\to}}
\DeclareMathOperator*{\wotlim}{{\mbox{\scshape wot}}-lim}
\newcommand{\wotto}{\stackrel{\textsc{wot}}{\to}}

\newcommand{\Xalg}[3]{#1(#2,#3)}
\newcommand{\Lalg}[2][G]{\Xalg{\mathfrak{L}}{#1}{#2}}
\newcommand{\Ralg}[2][G]{\Xalg{\mathfrak{R}}{#1}{#2}}

\def\dc#1{\expandafter\def\csname#1\endcsname{\mathcal{#1}}}
\def\db#1{\expandafter\def\csname b#1\endcsname{\mathbb{#1}}}

\def\loopy#1#2{\def#1##1{\def\next{#2{##1}#1}\ifx##1\relax\let\next\relax\fi\next}}

\loopy{\makemathcals}{\dc}\loopy{\makemathbbs}{\db}

\makemathbbs  CRNQZTDFPE\relax
\makemathcals QWERTYUIOPASDFGHJKLZXCVBNM\relax

\begin{document}

\title[Commutants of Weighted Shift Directed Graph Algebras]{Commutants of Weighted Shift Directed Graph Operator Algebras}

\author[D.~W.\ Kribs]{David W. Kribs}
\address{Department of Mathematics \& Statistics, University of Guelph, Guelph, ON, Canada N1G 2W1}
\author[R.~H.\ Levene]{Rupert H. Levene}
\address{School of Mathematics \& Statistics, University College Dublin, Belfield, Dublin 4, Ireland}
\author[S.~C.\ Power]{Stephen C. Power}
\address{Department of Mathematics \& Statistics, Lancaster University, Lancaster, U.K., LA1 4YF}


\begin{abstract}
We consider non-selfadjoint operator algebras $\Lalg\lambda$ generated by weighted creation operators on the Fock-Hilbert spaces of countable directed graphs~$G$. These algebras may be viewed as noncommutative generalizations of weighted Bergman space algebras, or as weighted versions of the free semigroupoid algebras of directed graphs. A complete description of the commutant is obtained
together with broad conditions that ensure the double commutant property.  It is also shown that the double commutant property may fail for $\Lalg{\lambda}$ in the case of the single vertex graph with two edges and a suitable choice of left weight function~$\lambda$.
\end{abstract}
\maketitle

\section{Introduction}

For over two decades, operator algebras associated with directed graphs and their generalizations have received intense interest in the operator algebra and mathematics community. This class of algebras includes many interesting examples, often with connections to different areas, such as dynamical systems, and at the same time is sufficiently broad that results for them have given insights to the general theory of operator algebras. The most fundamental non-selfadjoint algebras in this class are the tensor algebras~\cite{muhlysolel,popescu} and free semigroupoid algebras of directed graphs~\cite{katsouliskribs,kribspower,kribssolel}, including free semigroup algebras~\cite{davidsonpitts,kennedy}.  Each of these is generated by creation operators on a Fock-type Hilbert space defined by the graph, and there is now an extensive body of work for these algebras. %

In this paper we consider weighted creation operator generalizations, in the weak operator topology (WOT) closed setting,  and we investigate their algebraic structure. The resulting {\it weighted shift directed graph algebras} $\Lalg{\lambda}$ may be viewed as the minimal generalization of two different classes of non-selfadjoint algebras: the free semigroupoid algebras of directed graphs on the one hand, and on the other, the classical unilateral weighted shift algebras associated with single variable weighted Bergman spaces.

The paper is organized as follows. In the next section we introduce the notation  $\lambda$, $\rho$,  for certain left and right weight functions for the path semigroupoid of a directed graph $G$, and define their associated weighted creation operators (which need not be bounded) and their respective operator algebras, $\Lalg{\lambda}$ and $\Ralg{\rho}$, on the Fock space $\H_G$. In the subsequent section we investigate the structure of the commutant algebra $\Lalg{\lambda}'$  and obtain its characterization under the natural condition (left-boundedness of $\lambda$) that all the weighted left creation operators are bounded. In the proof we identify a simple commuting square condition that relates the left weight $\lambda$ to a particular right weight $\rho$ which is relevant to the commutant, and we exploit this to show that $\Lalg{\lambda}' =\Ralg{\rho}$ for this right weight.
In the fourth section we investigate the double commutant $\Lalg{\lambda}''$ and obtain broad conditions
which ensure the double commutant property $\Lalg{\lambda}''=\Lalg{\lambda}$.

A range of illuminating examples is also given. In particular, for the single vertex graph with two edges it is shown that there exist left-bounded weights $\lambda$ for which $\Lalg{\lambda} '' = \B(\H_G)$. On the other hand, for the directed $2$-cycle graph, with two vertices and two edges,  necessary and sufficient conditions are obtained for the double commutant property.

Our focus here is on the analysis of generalized weighted shifts and the non-selfadjoint operator algebras they generate, in a setting that embraces both commutative and non-commutative versions, and is built upon the contemporary directed graph operator algebra framework. In fact the first foray in this direction for single vertex directed graphs gave sufficient conditions for the determination of the commutant and for reflexivity~\cite{kribs04},
the basic general goals being to extend results from the single variable commutative case and to expose new phenomena in the non-commutative directed graph setting. Our concern in the present paper is to characterize  commutants for the left and right algebras by identifying explicit conditions at the level of weighted graphs. 
 It would be interesting to connect this double commutant investigation with the recent work \cite{marcouxmastnak} on a general double commutant
theorem for non-selfadjoint algebras, and with recent work on weighted Hardy algebras of correspondences~\cite{dor-on,muhlysolel-weighted}.

We leave the natural problems of invariant subspace structure and reflexivity  for these algebras for investigation elsewhere. It should be possible to identify a large class of these algebras as being reflexive, and in doing so, extend results from the case of weighted Bergman spaces~\cite{Shields74} and partial results from the weighted free semigroupoid algebra case~\cite{kribs04}. Additionally, non-reflexive examples have not yet been constructed in the non-commutative case. This should also be possible with extended notions of strictly cyclic weighted shifts to our setting.

\section{Weighted Shift Directed Graph Algebras}

Let $G$ be a countable
directed graph with edge set $E(G)=\{e,f,\ldots\}$ and vertex set $V(G)=\{x,y,\ldots\}$. 
We will write $G^+=\{ u,v,w,\ldots \}$ for the set of finite paths in $G$, including the vertices regarded as paths of length~$0$. Note that if~$G$ is finite (by which we mean that both $V(G)$ and~$E(G)$ are finite), then the set $\{w\in G^+\colon |w|< k\}$ is finite for each~$k\ge1$, where~$|w|$ denotes the length of a path~$w$. We write $s(w)$ and $r(w)$ for the source and range vertices of a path~$w$; in particular, if~$x\in V(G)$, then $r(x)=x=s(x)$. We will also take a right to left orientation for path products, so that $w = r(w) w s(w)$ for all $w\in G^+$, and for $v,w\in G^+$ we have $wv\in G^+$ if and only if $s(w)=r(v)$.

To each such graph $G$ we associate the Hilbert space $\H_G=\ell^2(G^+)$, called the Fock space of $G$, with canonical orthonormal basis $\{ \xi_v : v\in G^+ \}$. The vectors $\xi_x$ for $x\in V(G)$ are called vacuum vectors. The left creation operators on $\H_G$ are the partial isometries defined as follows: $L_w \xi_v = \xi_{wv}$ whenever $wv\in G^+$, and $L_w \xi_v = 0$ otherwise. (These operators may also be viewed as generated by the left regular representation of the  path semigroupoid of the graph.)  Pictorially, as an accompaniment to the directed graph, one can view the actions of the generators $L_e$ as tracing out downward tree structures that lay out the basis vectors for $\H_G$. One tree is present for each vertex $x$ in $G$, with the top tree vertex in each tree corresponding to a vacuum vector $\xi_x$, and the tree edges corresponding to the basis pairs $(\xi_w, \xi_{ew})$.
 
We call a function $\lambda\colon G^+\times G^+\to [0,\infty)$ a
\emph{left weight} on~$G$ if
\begin{enumerate}
\item $\lambda(v,w)>0\iff wv\in G^+$; and
\item  $\lambda$ satisfies the \emph{(left) cocycle condition}:
\[\lambda(v,w_2w_1) =  \lambda(w_1v,w_2)\lambda(v,w_1)\] for all $v,w_1,w_2\in G^+$ with $w_2w_1v\in G^+$.
\end{enumerate}
Note that if~$v\in G^+$, then $r(v)\in V(G)$ satisfies $r(v)=r(v)^2$, hence $\lambda(v,r(v))=\lambda(v,r(v))^2$ and so $\lambda(v,r(v))=1$. In particular, for $x\in V(G)$ we have $x=s(x)=r(x)$ and so $\lambda(s(x),x)=\lambda(x,r(x))=1$.
Note also that the edge weights $\lambda(v,e)$ (where $ev\in G^+$ and~$e\in E(G)$) determine the entire function $\lambda$ through the left cocycle condition. Indeed, if we attach the weight $\lambda(v,e)$ to the edge in the Fock space tree corresponding to the move $\xi_v\mapsto \xi_{ev}$ defined by $L_e\xi_v = \xi_{ev}$, then we can view $\lambda(v,w)$ (when non-zero) as the product of the individual weights one crosses when moving from $\xi_v$ to $\xi_{wv}$ in that tree. See subsequent sections for more discussion on this ``forest'' perspective.   

Given such a left weight~$\lambda$, we define (by a mild abuse of notation)
\[\lambda(w)=\sup_{v\in G^+} \lambda(v,w)\in [0,\infty]\]
for each $w\in G^+$. We say that~$\lambda$ is \emph{left-bounded at~$w$} if 
$\lambda(w) < \infty$, and that~$\lambda$ is \emph{left-bounded} if this condition holds for all $w\in G^+$. The cocycle condition gives $\lambda(w_2w_1)\leq  \lambda(w_2)\lambda(w_1)$ whenever $w_2w_1\in G^+$, so $\lambda$ is left-bounded if and only if $\lambda$ is left-bounded at every edge~$e\in E(G)$.

If $\lambda$ is left-bounded at~$w$, then we define the weighted left shift operator $L_{\lambda,w}\in \B(\H_G)$ to be the continuous linear extension of
\[
L_{\lambda,w} \, \xi_v =
\begin{cases}
\lambda(v,w) \,\xi_{wv} &\text{if $wv\in G^+$}\\
0 &\text{otherwise.}  
\end{cases}
\]
Since $\lambda(w)<\infty$, it is easy to see that this gives a well-defined operator with $\|L_{\lambda,w}\|=\lambda(w)$. Moreover, if~$\lambda$ is left-bounded, then by the cocycle condition we see that $w\mapsto L_{\lambda,w}$ is a semigroupoid homomorphism: \[L_{\lambda,w_2w_1}=L_{\lambda,w_2}L_{\lambda,w_1}\quad \text{whenever $w_2w_1\in G^+$}.\]

We remark that one could also consider complex-valued left weight functions rather than weights taking non-negative values only. However, the corresponding weighted left shift operators would be jointly unitarily equivalent to weighted shift operators defined by a non-negative weight function. To see this, consider a complex-valued left weight $\mu\colon G^+\times G^+\to \bC$, by which we mean that $\mu(v,w)\ne0\iff wv\in G^+$ and $\mu$ satisfies the left cocycle condition. Define corresponding weighted left shifts $L_{\mu,w}$ exactly as above, let $\lambda\colon G^+\times G^+\to [0,\infty)$ be the non-negative left weight $\lambda(v,w)=|\mu(v,w)|$ and consider $\beta\colon G^+\times G^+\to \bT$, $\beta(v,w)=\tfrac{\lambda(v,w)}{\mu(v,w)}$ when $wv\in G^+$, and $\beta(v,w)=1$ otherwise. Note that~$\beta$ then satisfies the left cocycle condition, so in particular we have $\beta(s(v),wv)=\beta(v,w)\beta(s(v),v)$ whenever $wv\in G^+$. The unitary operator $U_\beta$ mapping $\xi_v$ to $\beta(s(v),v)\xi_v$  satisfies
\[ U_\beta L_{\mu,w}\xi_v = \mu(v,w)\beta(s(v),wv)\xi_{wv}=\beta(s(v),v)\lambda(v,w)\xi_{wv} = L_{\lambda,w}U_\beta\xi_v
\]
whenever $wv\in G^+$, hence $U_\beta L_{\mu,w}=L_{\lambda,w}U_\beta$, i.e., $L_{\mu,w}=U_\beta^*L_{\lambda,w}U_\beta$.

Observe also that the requirement that $\lambda(v,w) \ne 0$ when $wv\in G^+$ is equivalent to requiring $L_{\lambda,w}$ to be injective on the set $\{\xi_v : wv\in G^+\}$, and is thus an assumption we build into the weights to avoid degeneracies in the analysis. Finally note that each operator $L_{\lambda,w}$ factors as a product of $L_w$ and a diagonal (with respect to the standard basis) weight operator, just as in the single variable case of~\cite{Shields74} which is recovered when~$G$ consists of a single vertex with a single loop edge.

We now define the algebras $\Lalg{\lambda}$ that we shall consider in the paper. In the case of the single vertex, single loop edge graph, these algebras include classical unilateral weighted shift algebras such as those associated with weighted Bergman spaces; see the survey article~\cite{Shields74} for an entrance point into the literature. The case of a single vertex graph and multiple loop edges was first considered along with some reflexivity type problems in~\cite{kribs04}. 

\begin{defn}
  If ${\lambda}$ is a left weight on a directed graph $G$, then we
  write $\Lalg{\lambda}$ for the WOT-closed unital operator
  algebra generated by the family of weighted left shift
  operators $\{ L_{\lambda,w} : w \in G^+,\ \lambda(w)<\infty\}$.
\end{defn}

\begin{rem}
  If (as is often the case below)~$\lambda$ is left-bounded, then the
  set \[\{L_{\lambda,w}\colon w\in V(G)\cup E(G)\}\] also
  generates~$\Lalg{\lambda}$ as a WOT-closed unital
  operator algebra.
\end{rem}

Let us call a strictly positive 
function~$\alpha\colon G^+\to (0,\infty)$ with $\alpha(x)=1$ for all
$x\in V(G)$ a \emph{path weight} on~$G$. For any such~$\alpha$, there is a corresponding
left weight~$\lambda_\alpha$ on~$G$ given by
\[ \lambda_\alpha(v,w)=
\begin{cases}
  \frac{\alpha(wv)}{\alpha(v)}&\text{if $wv\in G^+$}\\
  0&\text{otherwise.}
\end{cases}\] Conversely, from any left weight~$\lambda$, we obtain a
corresponding path weight $\alpha_\lambda\colon v\mapsto
\lambda(s(v),v)$, and these correspondences are inverses of one
another. This observation allows us to easily construct examples of
left weights.

The left-handed notions above have right-handed counterparts which will
play an important role in describing commutants. A \emph{right weight}
on~$G$ is a function~$\rho\colon G^+\times G^+\to [0,\infty)$
satisfying $\rho(v,u)>0\iff vu\in G^+$ and the \emph{(right) cocycle
  condition}
\[ \rho(v,u_1u_2)=\rho(vu_1,u_2)\rho(v,u_1)\] for all $v,u_1,u_2\in
G^+$ with $vu_1u_2\in G^+$. We then have $\rho(v,s(v))=1$ for all $v\in
G^+$. We write $\rho(u)=\sup_v\rho(v,u)$, and say $\rho$ is
\emph{right-bounded at~$u$} if $\rho(u)<\infty$. We may then consider
the weighted right shift operator~$R_{\rho,u}\in \B(\H_G)$ (with
$\|R_{\rho,u}\|=\rho(u)$) which satisfies the defining equation
\[ R_{\rho,u}\xi_v=
  \begin{cases}
\rho(v,u) \,\xi_{vu} &\text{if $vu\in G^+$}\\
0 &\text{otherwise.}  
\end{cases}
\]
We have $\rho(u_1u_2)\leq \rho(u_2)\rho(u_1)$, and
$R_{\rho,u_1u_2}=R_{\rho,u_2}R_{\rho,u_1}$ whenever $u_1u_2\in G^+$
and $\rho$ is right-bounded at $u_1$ and at~$u_2$. 

\begin{defn}
We write $\Ralg{\rho}$
for the WOT-closed unital operator algebra generated
by 
$
\{R_{\rho,u}\colon u\in G^+,\ \rho(u)<\infty\}
$. 
\end{defn}
A right weight is \emph{right-bounded} if it is right-bounded at
every~$u\in G^+$. Finally, we observe that 
\[\rho_\alpha(v,u)=
\begin{cases}
  \frac{\alpha(vu)}{\alpha(v)}&\text{if $vu\in G^+$}\\
  0&\text{otherwise}
\end{cases}\] defines a one-to-one correspondence between path
weights~$\alpha$ and right weights~$\rho=\rho_\alpha$.

\begin{rem}
Each of these right-handed definitions may be derived by applying the
corresponding left-handed definition to the opposite graph of~$G$, and making appropriate identifications.
Note that the suprema defining $\lambda(u)$ and $\rho(u)$ are taken over the first argument, in $\lambda(\cdot,u)$ and $\rho(\cdot,u)$, and so in particular the notion of right-boundedness for a left weight function does not arise. A path weight~$\alpha$, on the other hand, may be said to be left (resp. right) bounded if the associated map~$\lambda_\alpha$ (resp.~$\rho_\alpha$) is left-bounded (resp.~right-bounded).
\end{rem}

\begin{rem} 
The weighted shift creation operators $L_{\lambda, e}$ for edges of
$e$, and also sums of these operators, are in fact special cases of a wide class of weighted shift operators defined on general countable trees, rather than our graph generated trees. The single operator theory for these general shifts, such as conditions for hyponormality and $p$-hyponormality, is developed in the recent book of Jablo´nski,  Jung, and  Stochel~\cite{jjs-12}.
\end{rem}

\begin{rem}
Muhly and Solel have recently defined weighted shift versions of the Hardy algebras~$H^\infty(E)$~\cite{muhlysolel-weighted} that can be associated with a correspondence $E$ (a self-dual right Hilbert $C^*$-module) over a $W^*$-algebra $M$. The Hardy algebras $\A=H^\infty(E)$ in fact provide  generalizations of the free semigroupoid graph algebras in which the self-adjoint (diagonal) subalgebra $\A \cap \A^*$ is no longer commutative. At the expense of a much higher level of technicality, the weighted shift versions of these Hardy algebras similarly extend the weighted shift directed graph algebras~$\Lalg{\lambda}$.
\end{rem}

\section{Commutant Structure}

Let~$\lambda$ be a left weight on~$G$, and let~$\rho$ be a right
weight on~$G$.
We say that the pair $(\lambda,\rho)$ satisfies the \emph{commuting square condition}  at $(w,u)\in G^+\times G^+$
if
  \[ \rho(wv,u)\lambda(v,w)=\lambda(vu,w)\rho(v,u) 
	\] 
for every $v\in  G^+$ with $wvu\in G^+$.	If~$\lambda$ is left-bounded at~$w$ and~$\rho$ is right-bounded
at~$u$, then a simple computation shows that this condition holds if and only if  \[R_{\rho,u}L_{\lambda,w}=L_{\lambda,w}R_{\rho,u}.\]

Recall that associated with the left weight $\lambda$
is a forest graph whose vertices are labelled by the elements of $G^+$ and whose edges $(v,ev)$ are labelled by the individual weights $\lambda(v,e)$ for $e\in E(G)$.
We may now  augment this  ${\lambda}$-labelled forest  by additional ${\rho}$-edges $(v,ve)$, which are labelled by the individual nonzero weights ${\rho}(v,e)$. The resulting labelled graph is the union of two labelled edge-disjoint forests which share the same vertex set. The commuting square condition can be viewed as a commuting square within this labelled graph, for the weights indicated in Figure~\ref{f:weights}. 

\begin{figure}
\centering
\includegraphics{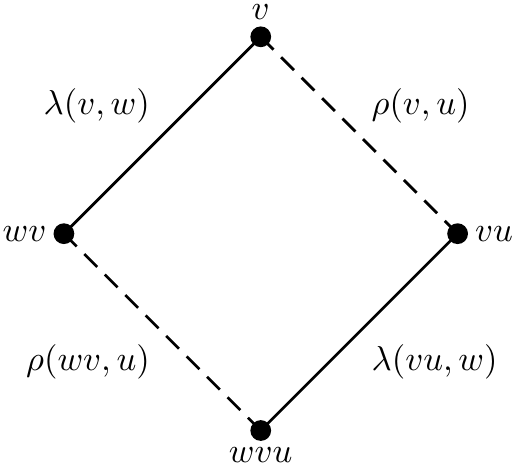}
\caption{The commuting square condition. Solid lines are paths made of edges labelled by $\lambda$-weights, and dashed lines are paths of edges labelled by $\rho$-weights.}
\label{f:weights}
\end{figure}

\begin{defn}
  Let~$\lambda$ be a left weight on~$G$. A right weight~$\rho$ on~$G$
  is a \emph{right companion} to~$\lambda$ if $(\lambda,\rho)$
  satisfies the commuting square condition at every~$(w,u)\in
  G^+\times G^+$. We call a right companion $\rho$ to~$\lambda$  \emph{canonical}
  if $\rho(r(e),e)=\lambda(s(e),e)$ for all $e\in E(G)$.
\end{defn}

\begin{prop}\label{prop:canonical}
  For any left weight~$\lambda$ on~$G$, there is a unique canonical
  right companion $\rho$ to~$\lambda$, namely $\rho=\rho_\alpha$ where
  $\alpha$ is the path weight with $\lambda=\lambda_\alpha$. Moreover,
  if~$\rho_1$ and~$\rho_2$ are both right companions to~$\lambda$,
  then $\Ralg{\rho_1}=\Ralg{\rho_2}$.
\end{prop}
\begin{proof}
  Let~$\alpha=\alpha_\lambda$ be the path weight given by
  $\alpha(v)=\lambda(s(v),v)$. Then $\lambda=\lambda_\alpha$ and by an
  easy calculation, the right weight $\rho_\alpha$ (defined in the
  previous section) is a canonical right companion to~$\lambda$. 

  If $\rho_1$ and $\rho_2$ are both right companions to~$\lambda$, then
  applying the commuting square condition for~$\rho_1$ and~$\rho_2$
  with $v=r(u)=s(w)$ shows that
  $q(u):=\frac{\rho_2(r(u),u)}{\rho_1(r(u),u)}>0$ satisfies
  $\rho_2(w,u)=q(u)\rho_1(w,u)$ for any $w,u$ with $wu\in G^+$. So
  $\rho_1$ is right-bounded at~$u$ if and only if~$\rho_2$ is
  right-bounded at~$u$, and in this case
  $R_{\rho_1,u}=q(u)R_{\rho_2,u}$; hence
  $\Ralg{\rho_1}=\Ralg{\rho_2}$. 

  If~$\rho_1$ and~$\rho_2$ are both canonical right companions
  to~$\lambda$, then $q(e)=1$ for all $e\in E(G)$. Applying the
  cocycle condition for~$\rho_1$ and~$\rho_2$ to the relation
  $\rho_2=q\cdot \rho_1$ shows that $q(u_1u_2)=q(u_2)q(u_1)$ whenever
  $u_1u_2\in G^+$; hence $q(u)=1$ for all $u\in G^+$, so
  $\rho_1=\rho_2$.
\end{proof}

For~$k\geq0$, let $Q_k$ be the orthogonal projection of~$\H_G$ onto
the closed linear span of $\{ \xi_v : |v|=k \}$. For $j\in\bZ$, define
a complete contraction $\Phi_j\colon \B(\H_G)\to \B(\H_G)$ by
  \[
  \Phi_j(X) = \sum_{m\geq \max\{ 0,-j\}} Q_m X Q_{m+j}.
  \]
  Also for $k\in\bN$, define $\Sigma_k\colon \B(\H_G)\to \B(\H_G)$ via the Cesaro-type sums
  \[  \Sigma_k(X) = \sum_{|j|<k} \Big( 1 - \frac{|j|}{k} \Big) \Phi_j(X).
  \]

 The following lemma is well-known. For completeness, we include a short proof.

  \begin{lem}\label{lem:Sigma-k}
    For $k\ge0$ and $X\in \B(\H_G)$, we have $\|\Sigma_k(X)\|\leq
    \|X\|$, and $\Sigma_k(X)$ converges to~$X$ in the strong operator
    topology as~$k\to \infty$.
  \end{lem}
  \begin{proof}
    Let $z\in \bT$ and let $U_z$ be the diagonal unitary operator on
    $\H_G$ for which $U_z\xi_v = z^{|v|}\xi_v$ for each $v \in
    G^+$. Then $U_zQ_sXQ_tU_z^* = z^{s-t}Q_sXQ_t$ for any $s,t\ge0$.
     Since $\sum_{\ell\ge0}Q_\ell=I$, it follows that for $j\in \bZ$, we have
    \[
    \Phi_j(X) = \int _{|z|=1}z^{j}U_zXU_z^*dz.
    \]
    Writing $F_k(z)=\sum_{j=-k}^k(1-\frac{|j|}{k+1})z^j$ for the usual Fej\'er kernel,
    we see that\[
    \Sigma_k(X) = \int_{|z|=1}F_{k-1}(z)U_zXU_z^*dz.
    \]
   Considering the scalars $\langle \Sigma_k(X)\xi, \zeta  \rangle$, for $\xi, \zeta \in \H_G$, and the fact that $\|F_{k-1}\|_{L^1(\bT)}=1$, it follows that    
          $\|\Sigma_k(X)\| \leq \|X\|$ for all $k$.
          
          Let $\xi \in \H_G$. Then
\[
    \|(X-\Sigma_k(X))\xi\| \leq \int_{|z|=1}F_{k-1}(z)\|(X - U_zXU_z^*)\xi\|dz.
    \]
The operators $U_z, U_z^*$ converge to the identity operator in the strong operator topology, as $z$ tends to $1$, and $F_k$ tends weak star to the unit point mass measure at $z=1$ as $k\to \infty$.  It follows that $\Sigma_k(X)\xi \to X\xi$ as~$k\to \infty$, and so $\Sigma_k(X)\sotto X$.
  \end{proof}

\begin{lem}\label{lem:Xf}
  Let $\rho$ be a right weight on~$G$ and suppose that $f\colon G^+\to
  \bC$ has the property that if $f(u)\ne0$, then $\rho(u)<\infty$.
  Let~$\H_0$ be the dense subspace of~$\H_G$ spanned by $\{\xi_v\colon
  v\in G^+\}$, and consider the sesquilinear form $A_f\colon
  \H_0\times\H_0\to \bC$ with \[A_f(\xi_v,\xi_w)=
  \begin{cases}
    f(u)\rho(v,u)&\text{if $w=vu$ for some $u\in G^+$}\\
    0&\text{otherwise}.
  \end{cases}\]
  If~$A_f$ is bounded on $\H_0\times \H_0$, then the operator $X_f\in
  \B(\H_G)$ implementing the continuous extension of~$A_f$
  to~$\H_G\times\H_G$ satisfies~$X_f\in \Ralg{\rho}$.
\end{lem}
\begin{proof}
  Let $(G_1,G_2,\dots)$ be a sequence of finite subgraphs of $G$ which
  increases to $G$; that is, $V(G_n)$ and $E(G_n)$ are finite sets for
  each~$n$, and $V(G_n)\uparrow V(G)$ and $E(G_n)\uparrow E(G)$.  For
  $n\ge1$, let $P_n$ be the projection onto the closure of the subspace
  of~$\H_G$ spanned by $\{\xi_v\colon v\in G_n^+\}$.  For $v,w\in
  G^+$, a calculation shows that $\langle
  P_n\Phi_j(X_f)P_n\xi_v,\xi_w\rangle=0$ unless $v,w\in G_n^+$ with
  $w=vu$ for some $u\in G_n^+$ with $|u|=-j$ and $\rho(u)<\infty$; and
  that in the latter case,
  \[ \langle
  P_n\Phi_j(X_f)P_n\xi_v,\xi_w\rangle=A_f(\xi_v,\xi_w)=f(u)\rho(v,u).\]
  It follows that $P_n\Phi_j(X_f)P_n= P_nF_{j,n} P_n$ where
  \[ F_{j,n}=
    \displaystyle\sum_{\substack{u\in G_n^+,\,|u|=-j,\\\rho(u)<\infty}} f(u)R_{\rho,u}.
    \] (We have $F_{j,n}=0$ if $j>0$.) Since $V(G_n)$ and $E(G_n)$ are
    finite, $F_{j,n}$ is a finite linear combination of operators
    $R_{\rho,u}$, so $F_{j,n}\in \Ralg{\rho}$.  Now $P_n\sotto I$,
    so \[\Phi_j(X_f)=\sotlim_{n\to \infty} P_n F_{j,n} P_n.\] In fact,
    we will shortly see that \[\Phi_j(X_f)=\wotlim_{n\to \infty}
    F_{j,n}.\] From this, it follows that $\Phi_j(X_f)\in \Ralg{\rho}$, so
    $\Sigma_k(X_f)\in \Ralg{\rho}$ for all $k\ge1$, allowing us to
    conclude, by Lemma~\ref{lem:Sigma-k}, that
    $X_f=\sotlim_{k\to\infty}\Sigma_k(X_f)\in \Ralg{\rho}$
    as desired.

    To see this, we will first show that $\{\|F_{j,n}\|\colon n\ge
    1\}$ is bounded. Note that for any~$v\in G^+$, we have
    the norm-convergent sums
  \[ X_f\xi_v=\sum_{w\in G^+} \langle X_f\xi_v,\xi_w\rangle \xi_w = \sum_{w\in G^+}A_f(\xi_v,\xi_w)\xi_w = \sum_{u\in G^+}f(u)\rho(v,u)\xi_{vu}.\]
  Moreover, 
  \[ F_{j,n}\xi_v=\sum_{\substack{u\in G_n^+,\,|u|=-j,\\\rho(u)<\infty}}f(u)\rho(v,u)\xi_{vu}\]
  so $\|F_{j,n}\xi_v\|\leq \|X_f\|$. For $i=1,2$, if $v_iu_i\in G^+$ and $|u_1|=|u_2|=-j$, then $v_1\ne v_2\implies v_1u_1\ne v_2u_2$.  It follows that $\{ F_{j,n}\xi_v\colon v\in G^+\}$ is a pairwise orthogonal family of vectors for each $n\ge1$, hence $F_{j,n}=\sum^{\oplus}_{v\in G^+} F_{j,n}\xi_v\xi_v^*$ and so
  \[ \|F_{j,n}\|=\sup_{v\in G^+} \|F_{j,n} \xi_v\|\leq \|X_f\|.\]
  Now $P_n^\perp:=I-P_n\sotto 0$ as $n\to \infty$, so  $P_nF_{j,n}P_n^\perp\sotto0$ and $P_n^\perp F_{j,n}\wotto0$ as $n\to \infty$. Hence
  \[ F_{j,n}=P_nF_{j,n}P_n+P_nF_{j,n}P_n^\perp+P_n^\perp F_{j,n}\wotto \Phi_j(X_f)  \quad\text{as $n\to \infty$},\]
which completes the proof.\end{proof}

\begin{rem}
  It is not difficult to see that if the function~$f$ in
  Lemma~\ref{lem:Xf} has finite support, then
  \[ X_f= \sum_{\substack{u\in G^+,\\\rho(u)<\infty}}f(u)R_{\rho,u}.\]
  Heuristically, it is useful to think of $X_f$ as the formal series
  given by this formula even when the support of~$f$ is
  infinite.
\end{rem}

\begin{lem}\label{lem:ker}
  Let~$\lambda$ be a left-bounded left weight on~$G$. If~$K\in \Lalg{\lambda}'$
  and $K\xi_x=0$ for all $x\in V(G)$, then $K=0$.
\end{lem}
\begin{proof}
  Given $w\in G^+$, consider $x=s(w)$. We have
  \[
  K\xi_w=\lambda(x,w)^{-1}KL_{\lambda,w}\xi_x=\lambda(x,w)^{-1}L_{\lambda,x}K\xi_x=0,\]
  so $K=0$.
\end{proof}

We are now ready to prove our main result. 

\begin{thm}\label{commutant_thm}
  If $\lambda$ is a left-bounded left weight on~$G$ and~$\rho$
  is its canonical right companion, then the commutant of
  $\Lalg{\lambda}$ coincides with~$\Ralg{\rho}$.
\end{thm}

\begin{proof}
The observations at the start of this section show that
$L_{\lambda,w}$ commutes with~$R_{\rho,u}$ whenever $w,u\in G^+$ and
$\rho(u)<\infty$. Hence $\Lalg{\lambda}^\prime$ contains $\Ralg{\rho}$.

To prove the other inclusion, begin by fixing $S \in \Lalg{\lambda}^\prime$. For~$u\in G^+$, consider the coefficients $a_u\in \bC$ defined by
\[ a_u=\langle S\xi_{r(u)},\xi_u\rangle.\]
Observe that for any~$x\in V(G)$, the operator~$L_{\lambda,x}=L_x$ is a projection with range spanned by $\{\xi_u\colon u\in r^{-1}(x)\}$, and $L_{x} S\xi_x=SL_{x}\xi_x=S\xi_x$. Hence 
\[ S\xi_x=\sum_{u\in G^+} \langle S\xi_x,\xi_u\rangle \xi_u=\sum_{u\in r^{-1}(x)} a_u\xi_u\]
with convergence in norm.

If~$v,w\in G^+$, then $\xi_v=\lambda(s(v),v)^{-1}L_{\lambda,v}
\xi_{s(v)}$ and $[S,L_{\lambda,v}]=0$, and $L_{\lambda,v}^*\xi_w=0$ unless $w=vu$ for some~$u\in G^+$, and $L_{\lambda,v}^*\xi_{vu}=\lambda(u,v)\xi_u$. Now $\frac{\lambda(u,v)}{\lambda(s(v),v)}=\frac{\rho(v,u)}{\rho(r(u),u)}$ by the commuting square condition, and it follows that
\begin{align*}
  \langle S\xi_v,\xi_w\rangle &=
  \begin{cases}
    \frac{\rho(v,u)}{\rho(r(u),u)}a_u &\text{if $w=vu$ for some $u\in G^+$}\\0&\text{otherwise}.
  \end{cases}
\end{align*}
In particular, if $a_u\ne 0$, then $\rho(u)<\infty$ since
\[ \|S\|\ge\sup_{\{v\in G^+\colon vu\in G^+\}}|\langle S\xi_v,\xi_{vu}\rangle | = \sup_{v\in G^+}\frac{\rho(v,u)}{\rho(r(u),u)}|a_u|=\frac{\rho(u)}{\rho(r(u),u)}|a_u|.\]

In view of this, if we define $f\colon G^+\to \bC$ by $f(u)=a_u
\rho(r(u),u)^{-1}$, then we may legitimately consider the bilinear form
$A_{f}\colon \H_0\times \H_0\to \bC$, defined as in
Lemma~\ref{lem:Xf}. 
Consider the operators $\Sigma_{k}(S)$. If~$v,w\in G^+$ and $\big||w|-|v|\big|<k$, then 
\begin{align*}
  \langle \Sigma_{k}(S)\xi_v,\xi_w\rangle 
&= \sum_{|j|<k}\left(1-\frac{|j|}k\right)\sum_{m\geq \max\{0,-j\}}\langle SQ_{m+j}\xi_v,Q_m\xi_w\rangle\\
&= \left(1-\frac{\big||w|-|v|\big|}k\right)\langle S\xi_v,\xi_w\rangle\\
&=
\begin{cases}
  \left(1-\frac{|w|-|v|}k\right)f(u)\rho(v,u)&\text{if $w=vu$ for some $u\in G^+$}\\0&\text{otherwise}
\end{cases}\\
&= \left(1-\frac{|w|-|v|}k\right)A_f(\xi_v,\xi_w).
\end{align*}
Hence for any $\xi,\eta\in \H_0$,  we have
$A_f(\xi,\eta)=\lim_{k\to \infty}\langle \Sigma_k(S)\xi,\eta\rangle$, so by Lemma~\ref{lem:Sigma-k},
\[ |A_f(\xi,\eta)|\leq \sup_{k\ge1}\|\Sigma_k(S)\|\,\|\xi\|\,\|\eta\|\leq \|S\|\,\|\xi\|\,\|\eta\|.\]
Thus~$A_f$ is bounded on $\H_0\times \H_0$. By Lemma~\ref{lem:Xf}, the bounded
linear operator $X=X_f$ implementing $A_f$ is in $\Ralg{\rho}$. So~$X\in
\Lalg{\lambda}'$, and (as above) we conclude that for $x\in V(G)$, the
vector $X\xi_x$ is in the closed subspace spanned by $\{\xi_u\colon
u\in r^{-1}(x)\}$. Moreover, for any $u\in r^{-1}(x)$ we have
\[ \langle X\xi_x,\xi_u\rangle = A_f(\xi_x,\xi_u)=f(u)\rho(x,u)=a_u=\langle S\xi_x,\xi_u\rangle.\]
So $X\xi_x=S\xi_x$ for all $x\in V(G)$. Since $X,S\in \Lalg{\lambda}'$, we have $K=X-S\in \Lalg{\lambda}'$ and $K\xi_x=0$ for all $x\in V(G)$. Hence $S=X$ by Lemma~\ref{lem:ker}, and so $S\in \Ralg{\rho}$, which completes the proof.
\end{proof}

\begin{rem}
This result generalizes and improves on a few previous results. The case of the single vertex and single edge graph yields classical single variable weighted shift operators, and there, the notion of right-boundedness simply corresponds to the weight sequence being bounded below. Hence this result generalizes the fundamental commutant theorem for weighted Bergman spaces $H^\infty(\beta)$ \cite{Shields74}. In the case of a single vertex graph with~$n$ edges and unit weights this result captures the commutant theorem for free semigroup algebras $\L_n$ \cite{davidsonpitts2}, and it improves on the commutant result of \cite{kribs04}, which established a special case of the theorem in the single vertex 
multi-edged weighted shift case. Finally, this result generalizes the commutant theorem for free semigroupoid algebras \cite{kribspower} which
are determined by general unweighted directed graphs.
\end{rem}

  \newcommand{\opposite}[1]{{{#1}^t}} 

\section{Double Commutant Theorems}

  When $\rho$ is right-bounded, we obtain the following mirror image of
  Theorem~\ref{commutant_thm} which may be established with a flipped
  version of the preceding proof. For brevity, we will instead pass
  to the opposite graph~$\opposite{G}$ of~$G$, which is essentially
  ``$G$ with the edges reversed''. More formally, we set $V(\opposite G)=V(G)$, $E(\opposite G)=E(G)$ and
  $(\opposite {G})^+=\{\opposite v\colon v\in G^+\}$, where
  $\opposite{(wv)}=\opposite{v}\opposite{w}$ for $wv\in G^+$, and
  $\opposite{u}=u$ for $u\in V(G)\cup E(G)$; the source and range maps
  for~$\opposite{G}$ are given by $ \opposite{s}(\opposite v)=r(v)$
  and $\opposite r(\opposite v)=s(v)$.

\begin{thm}
  If~$\lambda$ is a left weight on~$G$ whose canonical right
  companion~$\rho$ is right-bounded, then the commutant of $\Ralg{\rho}$
  coincides with $\Lalg{\lambda}$.
\end{thm} 

\begin{proof}
  Let~$\opposite{G}$ be the opposite graph of $G$ and let~$\opposite
  \rho(\opposite v,\opposite u)={\rho}(v,u)$ for $v,u\in G^+$; since
  $\rho$ is a right-bounded right weight on~$G$, it follows that
  $\opposite\rho$ is a left-bounded left weight on~$G^t$. A
  calculation using the path weight associated with $\lambda$
  and~$\rho$ shows that the canonical right companion
  to~$\opposite\rho$ is the right weight~$\opposite{\lambda}$ on~$G^t$
  given by $\opposite\lambda(v^t,w^t)=\lambda(v,w)$. Let $U\colon
  \H_{\opposite G}\to \H_G$ be the unitary with $U\xi_{\opposite
    v}=\xi_v$. For $u,w\in G^+$ with $\lambda(w)<\infty$, by checking
  values on basis vectors we see that
  \[UL_{\rho^t,\opposite u}U^*=R_{\rho,u} \qand
  UR_{\lambda^t,\opposite w}U^*=L_{\lambda,w},\] so
  \[U\Lalg[\opposite{G}]{\opposite{{\rho}}} U^*=\Ralg{\rho} \qand 
  U \Ralg[\opposite{G}]{\opposite{{\lambda}}}U^*=\Lalg{\lambda}.\] By
  Theorem~\ref{commutant_thm}, 
  $\Lalg[\opposite{G}]{\opposite{{\rho}}}'=\Ralg[\opposite{G}]{\opposite{\lambda}}$, hence
  \begin{align*}
    \Ralg{\rho}'=(U\Lalg[\opposite{G}]{\opposite{{\rho}}}U^*)'&=U\Lalg[\opposite{G}]{\opposite {\rho}}'U^*\\&=U\Ralg[\opposite{G}]{\opposite{{\lambda}}}U^*=\Lalg{\lambda}.\qedhere
  \end{align*}
\end{proof}

Combining the previous two results leads us to the following double commutant theorem.

\begin{thm}\label{preconj}
If~$\lambda$ is a left-bounded left weight on~$G$ whose canonical right companion~$\rho$ is right-bounded, then  $\Lalg{\lambda}$ coincides with its double commutant: \[\Lalg{\lambda}'' = \Lalg{\lambda}.\] 
\end{thm}

  If~$\alpha\colon G^+\to (0,\infty)$ is any path weight with
  $\sup_v\alpha(v)<\infty$ and $\inf_v\alpha(v)>0$, then plainly
  $\lambda_\alpha$ is left-bounded and its canonical right companion
  $\rho_\alpha$ is right-bounded, giving a large class of weights
  satisfying the hypotheses of this result. In particular, if
  $|G^+|<\infty$ (i.e., if $G$ is a finite \emph{acyclic} directed graph), then
  for any left weight $\lambda$ on~$G$, we see that $\Lalg{\lambda}$ is an
  algebra of $n\times n$ matrices with $\Lalg{\lambda}''=\Lalg{\lambda}$,
  where $n=|G^+|$.

  On the other hand, there are many weights which satisfy the hypotheses
  of Theorem~\ref{preconj} but violate these boundedness
  conditions for the path weight $\alpha$. For example, if $G^+$ contains paths of arbitrary length
  and $\alpha(v)=f(|v|)$ where $f\colon \bN_0\to (0,\infty)$ is any
  decreasing function with~$f(0)=1$ and $f(k)\to 0$ as $k\to \infty$,
  then $\lambda_\alpha$ is left-bounded and $\rho_\alpha$ is
  right-bounded, but $\inf_v\alpha(v)=0$.
 
We now show one way to weaken the hypotheses in
Theorem~\ref{preconj}, at least if~$G$ is a finite directed
graph. We first require a technical lemma.

\begin{lem}\label{lem:phi}

  Let~$\rho$ be a right weight on~$G$ and let $u\in G^+$ with
  $\rho(u)<\infty$.  Let $k\ge0$ and let $X\in \B(\H_G)$. If
  $[X,R_{\rho,u}]=0$, then $[\Sigma_{k}(X),R_{\rho,u}]=0$.
\end{lem}
\begin{proof}
  By calculating values on canonical basis vectors, we observe that
  \[ Q_mR_{\rho,u}=
  \begin{cases}
    R_{\rho,u}Q_{m-|w|}&\text{if $m\ge|u|$}\\
    0&\text{if $0\leq m<|u|$.}
  \end{cases}
  \]
  So if $[X,R_{\rho,u}]=0$, then
  \[ 
  \Phi_j(X)R_{\rho,u}=\sum_{m\geq
    \max\{|u|,|u|-j\}}R_{\rho,u}Q_{m-|u|}XQ_{m+j-|u|}=R_{\rho,u}\Phi_j(X).
  \]
  Since~$\Sigma_{k}(X)$ is a linear combination of 
  the operators $\Phi_j(X)$, which all commute with $R_{\rho,u}$, we see that $\Sigma_{k}(X)$ commutes with $R_{\rho,u}$.
\end{proof}

For any right weight~${\rho}$ on~$G$, let us write
\[G^+_{\rho}=\{u\in G^+\colon {\rho}(u)<\infty\}.\] Since $\rho(x)=1$ for $x\in V(G)$, we have $V(G)\subseteq G^+_\rho$. Moreover, since  $\rho(vw)\leq \rho(v)\rho(w)$ for any $vw\in G^+$, we see that $G^+_\rho$ is a subsemigroupoid in~$G^+$. 

\begin{thm}\label{thm:tails}
  Let~${\lambda}$ be a left-bounded left weight on a finite directed
  graph~$G$, with canonical right companion~$\rho$. If
  \begin{equation}\label{eq:tails}
    \forall\,v\in G^+\ \exists\,u_v\in  G^+_\rho\colon vu_v\in G^+_\rho,
  \end{equation}
  then $\Lalg{\lambda}''=\Lalg{\lambda}$.
\end{thm}
\begin{proof}
  It suffices to show that $\Lalg{\lambda}''\subseteq \Lalg{\lambda}$; equivalently (by Theorem~\ref{commutant_thm}) that $\Ralg{\rho}'\subseteq \Lalg{\lambda}$.
  Suppose that~$T\in\Ralg{\rho}'$. 
  If~$v\in G^+$, then $R_{\rho,s(v)}=R_{s(v)}$ is a projection
  in~$\Ralg{\rho}$ with $\xi_v=R_{s(v)}\xi_v$ and $[T,R_{s(v)}]=0$, from
  which it follows that $T\xi_v\in R_{s(v)}\H_G$. Moreover, if $u\in
  G^+_\rho$, then the restriction of $R_{\rho,u}^*R_{\rho,u}$ to
  $R_{r(u)}\H_G$ is an injective diagonal operator since it maps
  $\xi_v$ to $\rho(v,u)^2\xi_v$ if $s(v)=r(u)$. 

  Now suppose $K\in \Ralg{\rho}'$ and $K\xi_x=0$ for all $x\in V(G)$; we
  claim that we necessarily have $K=0$. To see this, let~$v\in G^+$,
  let~$u_v$ be as in Eq.~(\ref{eq:tails}) and note that $r(u_v)=s(v)$,
  since $vu_v\in G^+$. Now
  \begin{align*}
    R_{\rho,u_v}^*R_{\rho,u_v}K\xi_v&=    R_{\rho,u_v}^*KR_{\rho,u_v}\xi_v =\rho(v,u_v)R^*_{\rho,u_v}K\xi_{vu_v} \\&= \rho(v,u_v)\rho(r(v),vu_v)^{-1}R_{\rho,u_v}^*R_{\rho,vu_v}K\xi_{r(v)}=0,
  \end{align*}
  so $K\xi_v=0$ by the observations of the previous paragraph,
  establishing the claim.

  Now let $T\in \Ralg{\rho}'$ be arbitrary. Since~$G$ is finite, for each $k\in \bN$ the set $\{w\in G^+\colon |w|<k\}$ is finite and we may consider the operator
  \[ p_k(T)=\sum_{\{w\in G^+\colon |w|<k\}} \left(1-\frac{|w|}k\right) a_w
  \lambda(s(w),w)^{-1}L_{\lambda,w}\] where $a_w=\langle
  T\xi_{s(w)},\xi_w\rangle$ for $w\in G^+$. Clearly, $p_k(T)\in
  \Lalg{\lambda}$. As observed above, we have $T\xi_x\in R_x\H_G$,
  so \[T\xi_x=\sum_{w\in s^{-1}(x)} a_w\xi_w.\]
  By Lemma~\ref{lem:phi}, the operators $\Sigma_k(T)$
  are in $\Ralg{\rho}'$. Hence $K=\Sigma_k(T)-p_k(T)\in \Ralg{\rho}'$, and a
  calculation gives $K\xi_x=0$ for all $x\in V(G)$. Hence
  $\Sigma_k(T)=p_k(T)\in \Lalg{\lambda}$. Since~$\Lalg{\lambda}$ is strongly closed and
  $\Sigma_k(T)\to T$ strongly, we obtain $T\in \Lalg{\lambda}$. Hence
  $\Ralg{\rho}'\subseteq\Lalg{\lambda}$ which completes the proof.
\end{proof}

We now give some examples illustrating this result.  
If~$u\in G^+$, it will be useful to write
\[ G^+u=\{vu\colon v\in G^+ \text{ and } vu\in G^+\}.\]

\begin{figure}
\centering
\raisebox{-.5\height}{\includegraphics{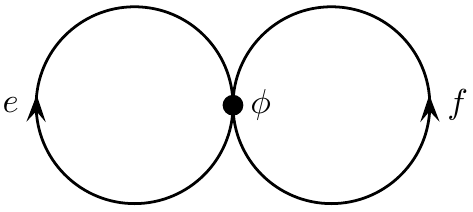}}\quad
\raisebox{-.5\height}{\includegraphics{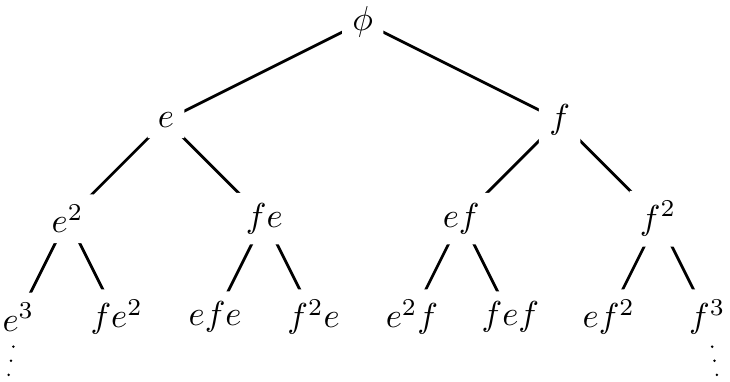}}%
\vspace*{-\bigskipamount}
\caption{$G$ and~$G^+$ in
  Examples~\ref{exa:rho(v)1} and~\ref{exa:rho(v)}}
\label{f:eg45}
\vspace*{2.5\bigskipamount}

\includegraphics{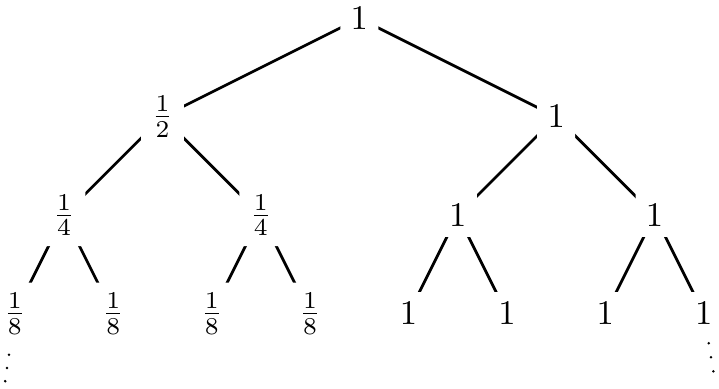}
\vspace*{-\bigskipamount}
\caption{The path weight~$\alpha$ considered in Example~\ref{exa:rho(v)1}}
\label{f:eg45a}
\vspace*{2.5\bigskipamount}

\raisebox{-.5\height}{\includegraphics{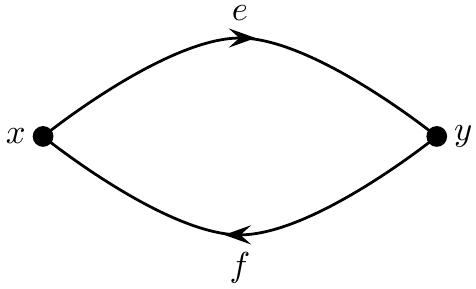}}\qquad\qquad\qquad
\raisebox{-.5\height}{\includegraphics{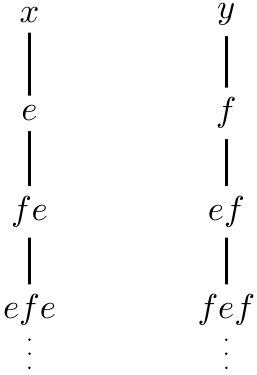}}
\vspace*{-\bigskipamount}
\caption{$G$ and~$G^+$ in
  Example~\ref{exa:twovertex}}
\label{f:eg46}
\vspace*{2.5\bigskipamount}

\includegraphics{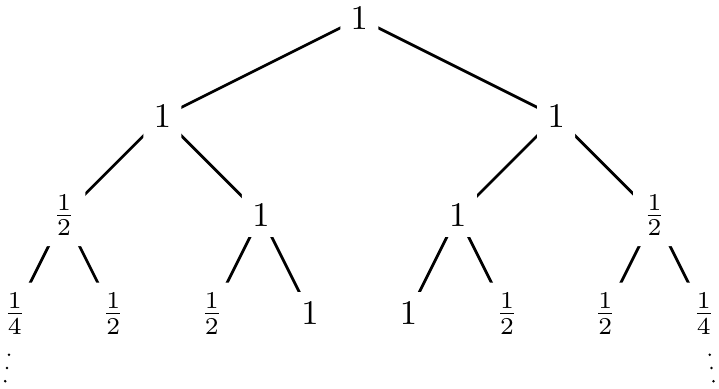}
\vspace*{-\bigskipamount}
\caption{The path weight~$\alpha$ considered in Example~\ref{exa:rho(v)}}
\label{f:eg47a}
\end{figure}

\begin{exa}\label{exa:rho(v)1}
  Let~$G$ be the directed graph with a single vertex~$\phi$ and two
  loop edges, $e$ and~$f$, so that
  $G^+=\{\phi,e,f,ee=e^2,ef,fe,f^2,eee=e^3,\dots\}$ and
  $s(w)=r(w)=\phi$ for every~$w\in G^+$ (see Figure~\ref{f:eg45}).  As
  indicated in Figure~\ref{f:eg45a}, we consider the path weight
  $\alpha\colon G^+\to (0,\infty)$ given by
  \[\alpha(v)=
  \begin{cases}
    2^{-|v|}&\text{if $v\in G^+e$}\\
    1&\text{otherwise},
  \end{cases}
  \]
  and let $\lambda=\lambda_\alpha$.  It is easy to check that the left
  weight~$\lambda$ is  left-bounded (in
  fact $\lambda(w)\leq 1$ for all $w\in G^+$); the canonical
  right companion of~$\lambda$ is~$\rho=\rho_\alpha$ by 
  Proposition~\ref{prop:canonical}. 

  We claim that while $\rho$ is not right-bounded, we
  have
  \[ G^+_{\rho}=\{\phi\}\cup G^+e = G^+\setminus G^+f,\] so that
  Eq.~(\ref{eq:tails}) holds with $u_v=e$ for all $v\in G^+$, hence
  $\Lalg{\lambda}''=\Lalg{\lambda}$.
  
  Let us check that $G^+_{\rho}$ is indeed of this form. We have
  ${\rho}(v,u)=\frac{\alpha(vu)}{\alpha(v)}$ for any $ v,u\in G^+$.  In
  particular, if $u\in G^+f$, then $\alpha(vu)=1$ and so $\rho(v,u)=2^{|v|}$ for $v\in G^+e$,
  so ${\rho}(u)=\infty$. On the other hand, if $u\in
  G^+\setminus G^+f$, then $\alpha(vu)=2^{-|vu|}\leq \alpha(v)$ for all~$v\in G^+$, so $\rho(u)\leq 1$.
\end{exa}

\begin{exa}
  \label{exa:twovertex}
  Let~$G$ be the directed $2$-cycle, so that $V(G)=\{x,y\}$
  and~$E(G)=\{e,f\}$ where $s(e)=r(f)=x$ and $s(f)=r(e)=y$ (see Figure~\ref{f:eg46}). For this
  particular graph~$G$, we will show that if~$\lambda$ is any
  left-bounded left weight on~$G$ with right companion~$\rho$, then
  $\Lalg{\lambda}''=\Lalg{\lambda}$ if and only if $G_\rho^+$ satisfies
  Eq.~(\ref{eq:tails}).

  Note that the edges in any path in~$G^+$ must
  alternate: \[G^+=\{x,y,e,f,ef,fe,efe,fef,efef,fefe,\dots\}.\] We
  first show that $ef\in G^+_\rho$. Let $v\in G^+\setminus V(G)$ with
  $vef\in G^+$.  Either $v=(ef)^{k}$ or $v=f(ef)^{k-1}=(fe)^{k-1}f$
  for some $k\ge1$. Let~$\alpha\colon G^+\to (0,\infty)$ be the path
  weight with $\lambda=\lambda_\alpha$ and $\rho=\rho_\alpha$. Since $(ef)^kef=(ef)^{k+1}=ef(ef)^k$, we have \[\rho((ef)^k,ef)=\frac{\alpha((ef)^{k+1})}{\alpha((ef)^k)} = \lambda((ef)^k,ef)\leq \lambda(ef)<\infty\]
  and since $f(ef)^{k-1}ef=f(ef)^k=(fe)^kf$, we have
  \[\rho(f(ef)^{k-1},ef)=\frac{\alpha((fe)^{k}f)}{\alpha((fe)^{k-1}f)}=\lambda((fe)^{k-1}f,fe)\leq \lambda(fe)<\infty,\]
  so $\rho(ef)<\infty$, i.e., $ef\in G^+_\rho$. Similarly, $fe\in
  G^+_\rho$. Since $G_\rho^+$ is a semigroupoid, we have $\langle ef,fe\rangle \subseteq G^+_\rho$ where $\langle
  ef,fe\rangle := V(G)\cup \{ (ef)^k,(fe)^k\colon k\ge1\}$.
  If~$G_\rho^+\supsetneq\langle ef,fe\rangle$, then
  $G_\rho^+$ contains an element of odd length. By symmetry, we may
  assume this is of the form $e(fe)^n$ for some $n\ge0$. We then also
  have $e(fe)^m=e(fe)^n(fe)^{m-n}\in G_\rho^+$ for any~$m> n$, so if
  we define $u_v$ for $v\in G^+$ by
  \[ u_v=
  \begin{cases}
    s(v)&\text{if $v\in \langle ef,fe\rangle$}\\
    (fe)^n&\text{if $v=e(fe)^k$ for some $k\ge0$}\\
    e(fe)^n&\text{if $v=f(ef)^k$ for some $k\ge0$,}
  \end{cases}
  \]
  then $u_v\in G^+_\rho$ and $vu_v\in G^+_\rho$ for all $v\in G^+$, so
  Eq.~(\ref{eq:tails}) holds and so $\Lalg{\lambda}''=\Lalg{\lambda}$ by
  Theorem~\ref{thm:tails}.
  
  On the other hand, if $G^+_\rho=\langle ef,fe\rangle$, then
  Eq.~(\ref{eq:tails}) does not hold, since if $|v|$ is odd and $u\in
  G^+_\rho$ with $vu\in G^+$, then $|vu|$ is also odd, so $vu\not\in
  G_\rho^+$. In this case, it is not difficult to see that the
  orthogonal projection $P$ onto the closed linear span of
  $\{\xi_v\colon \text{$v=x$ or $v=(fe)^k$, $k\ge1$}\}$ commutes with
  $R_{\rho,u}$ for $u\in \{x,y,ef,fe\}$, hence $P\in
  \Ralg{\rho}'=\Lalg{\lambda}''$. If~$T\in \Lalg{\lambda}$, then $\langle
  T\xi_x,\xi_x\rangle = \langle T\xi_f,\xi_f\rangle$. Since
  $P\xi_x=\xi_x$ and $P\xi_f=0$, we have $P\not\in \Lalg{\lambda}$. So
  $\Lalg{\lambda}''\ne \Lalg{\lambda}$.

  We note that it is indeed possible for Eq.~(\ref{eq:tails}) to fail
  for this graph~$G$. For example, let $\alpha\colon G^+\to
  (0,\infty)$ with $\alpha(v)=1$ for all $v\in s^{-1}(x)$, and
  $\alpha(v)=2^{f(|v|)}$ for $v\in s^{-1}(y)$ where $f\colon \bN_0\to
  \bZ$ is a function with $f(0)=0$, and $f(n+1)\in \{f(n)-1,f(n)+1\}$
  for all $n\in \bN_0$ and with $\sup_n f(n)=\infty$ and $\inf_n
  f(n)=-\infty$. One may then check that the left weight
  $\lambda_\alpha$ is left-bounded, and that its canonical right
  companion~$\rho$ satisfies $G_\rho^+=\langle ef,fe\rangle$, so
  Eq.~(\ref{eq:tails}) fails.
\end{exa}

\begin{exa}\label{exa:rho(v)}
  For a general left-bounded weight~$\lambda$, the double commutant
  property for $\Lalg{\lambda}$ can fail very badly. For example, let~$G$
  again be the directed graph with a single vertex~$\phi$ and two loop
  edges $e$ and~$f$, and let us now define a path weight $\alpha\colon
  G^+\to (0,\infty)$ recursively by setting
  $\alpha(\phi)=\alpha(e)=\alpha(f)=1$, and
  \begin{align*}
    \alpha(ewe)&=\tfrac12 \alpha(we),&\alpha(fwf)&=\tfrac 12 \alpha(wf),\\
    \alpha(ewf)&=\alpha(wf), &\alpha(fwe)&=\alpha(we).
  \end{align*}
  This is illustrated in Figures~\ref{f:eg45} and~\ref{f:eg47a}.
  Take $\lambda=\lambda_\alpha$ and $\rho=\rho_\alpha$. Observe
  that~${\lambda}$ is a left-bounded left weight
  since~$\alpha(wv)\leq \alpha(v)$ for all $w,v\in G^+$. For any~$k\in
  \bN$ and $w\in G^+$,
  \[ {\rho}(we)\geq
  {\rho}(f^k,we)=\frac{\alpha(f^kwe)}{\alpha(f^k)}=2^{k-1}
  \alpha(we)\to \infty\text{ as~$k\to \infty$,}\] hence ${\rho}(we)=\infty$;
  by symmetry, ${\rho}(wf)=\infty$. Hence $
  G^+_{\rho}=\{\phi\}$. Since $R_{\rho,\phi}=R_\phi=I$, we have $\Ralg{\rho}=\bC I$
  and so $\Lalg{\lambda}''=\Ralg{\rho}'=\B(\H_G)\ne \Lalg{\lambda}$.
\end{exa}

The commutant result yields other structural results on the algebras,
such as the following. %

\begin{cor}
  Let~${\lambda}$ be a left-bounded left weight on~$G$ with canonical right companion~$\rho$.
  If either $\rho$ is right-bounded, or~$G$ is finite and $G^+_{\rho}$ satisfies Eq.~(\ref{eq:tails}), then $\Lalg{\lambda}$ is inverse closed.
\end{cor} 
\begin{proof}
  This is a well-known property of commutants: if $A\in\Lalg{\lambda}
  = \Ralg{\rho}'$ is invertible in $\B(\H_G)$, then for all
  $R\in\Ralg{\rho}$, $A^{-1}R = A^{-1} RAA^{-1} = RA^{-1}$, and hence
$A^{-1}\in\Ralg{\rho}' = \Lalg{\lambda}$. 
\end{proof}

\begin{cor}
  If~${\lambda}$ is a left-bounded left weight on~$G$ whose canonical
  right companion~$\rho$ is right-bounded, then every normal element
  of $\Lalg{\lambda}$ lies in the SOT-closure of the linear
  span of the projections $L_{\lambda,x}$ for $x\in V(G)$.
  \end{cor} 

\begin{proof}
  Let $N$ be a normal element of $\Lalg{\lambda}$. Set $a_x = \langle N\xi_x,\xi_x\rangle$ for $x\in V(G)$ and let~$M$ be the SOT-convergent sum $M=\sum_{x\in V(G)}a_xL_{\lambda,x}$. Observe that each $\xi_x$ is an eigenvector for $\Lalg{\lambda}^*$, as for all $u\in G^+\setminus\{x\}$ and $A\in \Lalg{\lambda}$, we have \[
  \langle A^*\xi_x, \xi_u\rangle =\tfrac1{{\rho}(r(u),u)} \langle\xi_x, A R_{\rho,u} \xi_{r(u)}\rangle=
 \tfrac1{{\rho}(r(u),u)}\langle R_{\rho,u}^*\xi_x,A\xi_{r(u)}\rangle=0.
\]
Thus $N^* \xi_x = \overline{a_x} \xi_x$, and by normality $N\xi_x = a_x \xi_x$. Hence for all $u\in G^+$, \[N\xi_u = \tfrac1{{\rho}(r(u),u)}N R_{\rho,u} \xi_{r(u)} = \tfrac1{{\rho}(r(u),u)}R_{\rho,u} N\xi_{r(u)} = a_{r(u)} \xi_u=M\xi_u,\] so $N=M$. 
\end{proof}

\begin{rem}
  Building on Theorems~\ref{preconj} and~\ref{thm:tails}, a natural open problem is to determine weighted graph conditions  that fully characterize when 
the algebra~$\Lalg{\lambda}$ and its double commutant~$\Lalg{\lambda}''$ coincide.
\end{rem}

\begin{ack} This work was partly
  supported by a London Mathematical Society travel grant. The first named author was
  partly supported by NSERC Discovery Grant 400160 and a University Research Chair at Guelph. We are grateful to the referee for helpful suggestions. 
\end{ack}

\bibliographystyle{plain}
\bibliography{myrefs}
\end{document}